\documentclass[11pt,leqno]{amsart}
\usepackage{amscd}
\usepackage{color}
\usepackage[symbol]{footmisc}
\usepackage{amssymb}
\usepackage{amsfonts}
\usepackage{latexsym}
\usepackage{verbatim}
\usepackage{braket}
\newcommand\norma[1]{\left\lVert#1\right\rVert}
\newcommand\abs[1]{\left\lvert#1\right\rvert}

\newcommand{\Z}{\mathbb{Z}}

\newcommand{\R}{\mathbb{R}}

\DeclareMathOperator{\ui}{\mathrm{i}} 
\DeclareMathOperator{\uj}{\mathrm{j}}
\DeclareMathOperator{\uk}{\mathrm{k}}
\renewcommand{\epsilon}{\varepsilon}
\renewcommand{\theta}{\vartheta}
\renewcommand{\phi}{\varphi}

\theoremstyle{plain}
\newtheorem{teor}{Theorem}
\newtheorem{prop}[teor]{Proposition}
\newtheorem{lem}[teor]{Lemma}
\newtheorem{cor}[teor]{Corollary}

\theoremstyle{remark}
\newtheorem*{oss}{Remark}

\theoremstyle{definition}

\title[HKT Calabi-Yau equation on tori fibrations]{The quaternionic Calabi conjecture on Abelian hypercomplex nilmanifolds viewed as tori fibrations}

\begin{document}
 
\thanks{This work was supported by GNSAGA of INdAM}
\address{Dipartimento di Matematica G. Peano \\ Universit\`a di Torino\\
Via Carlo Alberto 10\\
10123 Torino\\ Italy}
\email{giovanni.gentili@unito.it} \email{luigi.vezzoni@unito.it}

\author{Giovanni Gentili and Luigi Vezzoni}
\date{\today}

\maketitle
\begin{abstract}
We study the quaternionic Calabi-Yau problem in HKT geometry, introduced by  Alesker and Verbitsky in \cite{Alesker-Verbitsky (2010)}, on $8$-dimensional $2$-step nilmanifolds $M$ with an Abelian hypercomplex structure. We show that on these manifolds the quaternionic Monge-Amp\`ere equation can always be solved for any data that is invariant under the action of a $3$-torus.
\end{abstract}

\section{Introduction}

Since Yau proved the Calabi-Yau conjecture in \cite{Y}, other Calabi-Yau-type problems have been introduced in various geometric contexts. 

In the present paper we focus on a generalization of the Calabi-Yau problem to HKT geometry, which was introduced by  Alesker and Verbitsky in \cite{Alesker-Verbitsky (2010)}. 

{\em HyperK\"ahler} manifolds {\em with torsion} (HKT) were introduced by Howe and Papadopoulos in \cite{Howe-Papadopoulos (2000)} in the framework known as `geometries with torsion'. In a nutshell, they can be thought of as hypercomplex manifolds admitting a special compatible Riemannian metric. 

A {\em hypercomplex manifold} is a $4n$-dimensional real manifold $M$ equipped with a  triple of complex structures  $ J_1,J_2,J_3$ satisfying the quaternionic relations
\begin{equation}\label{HC}
J_1J_2=-J_2J_1=J_3\,.
\end{equation} 

If $g$ is a Riemannian metric on $M$ that is compatible with $J_1,J_2,J_3$, then  
$(M,J_1,J_2,J_3,g)$ is usually called a \emph{hyperHermitian} manifold. According to the classical definition, a hyperHermitian manifold $(M,J_1,J_2,J_3,g)$ is said HKT  if there exists an affine connection $\nabla$  on $M$ which preserves the hyperHermitian structure and has totally skew-symmetric torsion. If such $\nabla$ exists, it is necessarily unique.  The existence of $\nabla$ can be characterized in terms of the differential equation 
$$
\partial \Omega=0\,,
$$
where $\partial$ is taken with respect to $J_1$\,,
\[
\Omega=\omega_{J_2}+\ui \omega_{J_3}\,,
\]
and 
$$
\omega_{J_r}(\cdot,\cdot)=g(J_r\cdot,\cdot)\,.
$$

In this context $\Omega$ is called the {\em HKT form} of the HKT structure, and one may think of it as the analogue of the fundamental form in K\"ahler geometry.   
The hypercomplex condition \eqref{HC} implies that $\Omega$
is of type $(2,0)$ with respect to $J_1$, and it satisfies
\begin{equation*}
\Omega(J_2\cdot,J_2\cdot)=\bar\Omega
\end{equation*}  
and 
\begin{equation*}
\Omega(X,J_2X)>0\,\quad\mbox{for every nowhere vanishing real vector field $X$ on $M$.}
\end{equation*}  
Moreover, $\Omega$ determines the metric $g$ via the relation 
$$
g(X,Y)={\rm Re}\,\Omega(X,J_2Y)\,,  \,\quad\mbox{for any real vector fields $X,Y$ on $M$.}
$$
An HKT structure can then be defined alternatively, as a hypercomplex structure together with an HKT form. 

\medskip 
In \cite{Alesker-Verbitsky (2010)} the authors introduced the following Calabi-Yau-type problem in HKT geometry.  
Let  $(M,J_1,J_2,J_3,\Omega)$ be a compact $4n$-dimensional HKT manifold for which the canonical bundle of $(M,J_1)$ is holomorphically trivial, and suppose $F\in C^{\infty}(M)$ is a function satisfying 
\begin{equation}\label{normalization}
\int_M (\mathrm{e}^{F}-1)\, \Omega^n \wedge \bar{\Theta}=0\,,
\end{equation}
where $\Theta$ is a non-vanishing holomorphic $(2n,0)$-form on $(M,J_1)$.
The {\em quaternionic Calabi-Yau problem} consists in finding an HKT form $\tilde \Omega$ on $ (M,J_1,J_2,J_3)$ such that  
\begin{equation}\label{CY}
\tilde{\Omega}^n=\mathrm{e}^F\Omega^n \,. 
\end{equation}
Just like the classical version, the quaternionic Calabi-Yau problem, too, can be rewritten in the form of a Monge-Amp\`ere equation. Indeed, results in \cite{BS} guarantee the unknown HKT form $\tilde\Omega$ can be written in terms of an HKT potential   $\phi \in C^\infty(M)$ as follows
$$
\tilde \Omega=\Omega+\partial\partial_{J_2}\varphi\,.
$$ 
Here $ \partial_{J_2} $ is the so-called \emph{twisted Dolbeault operator}
\[
\partial_{J_2}=-J^{-1}_2\bar{\partial}J_2
\]
and the complex structure $ J_2 $ acts on $ k $-forms $ \alpha $ by
\[
J_2\alpha(X_1,\dots,X_k)=(-1)^k\alpha(J_2X_1,\dots,J_2X_k)\,.
\]

Equation \eqref{CY} reads, in terms of $\phi$ and  $F$, 
\begin{equation}
\label{eq_quaternionic_Calabi_conjecture}
(\Omega+\partial \partial_{J_2}\phi)^n=\mathrm{e}^F\Omega^n\,.
\end{equation}

It has been conjectured in \cite{Alesker-Verbitsky (2010)} that the above equation 
can always be solved under assumption \eqref{normalization}. The authors of the same paper propose the continuity method as a natural approach to attack the problem, much in the same spirit of Yau's proof of the Calabi conjecture \cite{Y}. The hard part in this line of thought is to establish a priori estimates. Alesker and Verbitsky \cite{Alesker-Verbitsky (2010)} showed the solution is unique up to an additive constant, and proved a $ C^0 $-estimate. The latter was later generalized by Alesker-Shelukhin \cite{Alesker-Shelukhin (2017)}, and then by Sroka \cite{Sroka} in a more general setting. Alesker gave evidence for believing the conjecture in \cite{Alesker (2013)}, where he proved that the quaternionic Monge-Amp\`ere equation has solutions if the manifold admits a flat hyperK\"ahler metric compatible with the underlying hypercomplex structure. 

\medskip 
The work of the present paper takes off where \cite{BFV,Buzano-Fino-Vezzoni (2015),FLSV,TW,TW2,V} stopped. Those articles studied the symplectic Calabi-Yau conjecture \cite{D,W} on torus fibrations in the case the problem's data admits certain symmetries. In the same spirit, we study the quaternionic Monge-Amp\`ere equation on compact quotients of  $8$-dimensional nilpotent Lie groups endowed with an {\em Abelian} HKT structure.  
  
By a result of Dotti and Fino \cite{Dotti-Fino (2000)} the only non-Abelian $8$-dimensional $2$-step nilpotent Lie groups admitting an Abelian hypercomplex structure are 
\begin{align*}
N_1=H_1(2)\times \R^3\,, && N_2=H_2(1)\times \R^2\,, && N_3=H_3(1)\times \R\,, 
\end{align*}
where $H_{i}(n)$ denotes the real ($i=1$), complex ($i=2$), and quaternionic ($i=3$) Heisenberg group.  Each $N_i$ contains a canonical co-compact lattice $\Gamma_i$,  and the nilmanifold $M_i=\Gamma_i\backslash N_i$, i.e. the quotient of $N_i$ by $\Gamma_i$, inherits the structure of a principal $T^3$-bundle over a $5$-dimensional torus $T^5$ and also an HKT structure $(J_1,J_2,J_3,g)$ (see section \ref{pre} for details). In view of \cite{BG} the nilmanifolds $M_i$ are not K\"ahlerian, since a compact nilmanifold admits a K\"ahler metric if and only if it is a torus.

Moreover, the canonical bundle of $(M_i,J_1)$ is holomorphically trivial \cite[Theorem 2.7]{BDV} and $M_i$ carries a left-invariant holomorphic volume form $\Theta$. 
Hence it is quite natural to wonder whether the Alesker-Verbitsky conjecture might hold on these spaces.  

\medskip 
Our main result is the following 

\begin{teor}\label{main}
The quaternionic Monge-Amp\`ere equation  \eqref{eq_quaternionic_Calabi_conjecture} on $(M_i,J_1,J_2,J_3,g)$ can be solved for every $T^3$-invariant map $F\in C^{\infty}(M_i)$ satisfying \eqref{normalization}.
\end{teor}
 
Since we are assuming $F$ is invariant under the action of the fibre $T^3$, it can be regarded as a smooth function on the base $T^5$. Furthermore, condition \eqref{normalization} can be written as 
\begin{equation}\label{normalization1}
\int_{T^5}(\mathrm{e}^F-1)dx^1\cdots dx^5=0\,.
\end{equation}
By imposing the same invariance property on the HKT potential $\varphi$, we reduce the quaternionic Monge-Amp\`ere equation on $(M_i,J_1,J_2,J_3,g)$ to  
 \begin{equation}\label{eq_equation}
(\phi_{11}+\phi_{22}+\phi_{33}+\phi_{44}+1)(\phi_{55}+1)-\phi_{15}^2-\phi_{25}^2-\phi_{35}^2-\phi_{45}^2=\mathrm{e}^F\,,
\end{equation}
where $\phi_{rs}$ denotes the second derivative of $\phi$ in the real coordinates $x^r,x^s \in \{x^1,\dots,x^5\}$ on $T^5$. Then we prove that equation  \eqref{eq_equation} has a solution $\varphi\in C^{\infty}(T^5)$ whenever $F$ satisfies  \eqref{normalization1}.  

The strategy for proving Theorem \ref{main} goes as follows:  in section  \ref{C^0} we prove the $C^0$-estimate for our equation. Then in section  \ref{Laplacian} we deduce an a priori $C^0$-estimate for the Laplacian of a solution to our equation, and in section  \ref{C^2,alpha} we achieve the $C^{2,\alpha}$-estimate by applying a general result of Alesker \cite{Alesker (2013)}. Eventually, we complete the proof in section \ref{Proof} by applying the continuity method.

\bigskip
\noindent {\bf Acknowledgements.} The authors are very grateful to Ernesto Buzano, Anna Fino, Alberto Raffero and 
 Simon Chiossi  for many useful conversations.

\section{Preliminaries}\label{pre}

Let $G$ be an $8$-dimensional Lie group with a left-invariant hypercomplex structure $(J_1,J_2,J_3)$ (every complex structure $J_i$ is left-invariant).   Assume that   $J_1$ is {\em Abelian}, meaning 
\[
[J_1X,J_1Y]=[X,Y], \quad \mbox{ for every }\,\, X,Y\in \mathfrak{g}\,,
\]
where $\mathfrak{g}$ is the Lie algebra of $G$. Recall that this is equivalent to requiring that the Lie algebra $\mathfrak{g}^{1,0}$ of left-invariant vector fields of type $(1,0)$ on $(G,J_1)$ is Abelian. It also implies that any left-invariant  $(p,0)$-form on $(G,J_1)$ is $\partial$-closed. 
If $g$ is a left-invariant Riemannian metric on $G$  compatible   
with $(J_1,J_2,J_3)$, the hyperHermitian structure $(J_1,J_2,J_3,g)$ is HKT because the corresponding form $\Omega$ is $\partial$-closed. 

\medskip
As we mentioned in the introduction, by
 \cite{Dotti-Fino (2000)} the only  $8$-dimensional nilpotent, non-Abelian, Lie groups carrying a left-invariant HKT structure $(J_1,J_2,J_3,g)$ such that every $J_i$ is Abelian are  
\begin{align*}
N_1=H_1(2)\times \R^3\,, && N_2= H_2(1)\times\R^2\,, && N_3=H_3(1)\times\R\,,
\end{align*} 
where  

\begin{align*}
H_1(2)&=\left\{ \begin{pmatrix}
1 & x^1 & x^4 & y^1\\
0 & 1 & 0 & x^3\\
0 & 0 & 1 & x^2\\
0 & 0 & 0 & 1
\end{pmatrix} \right\},\quad 
H_2(1)=\left\{ \begin{pmatrix}
1 & x^1+\ui x^2 & y^3+\ui y^2 \\
0 & 1 & x^4+\ui x^3\\
0 & 0 & 1
\end{pmatrix} \right\},\\
H_3(1)&=\left\{ \begin{pmatrix}
1 & q & h-\frac{1}{2}q\bar{q}\\
0 & 1 & -\bar{q}\\
0 & 0 & 1
\end{pmatrix} \mid q=x^1+\ui x^4+\uj x^3+\uk x^2, h=\ui y^3+\uj y^2+\uk y^1 \right\}.
\end{align*}
Above, $x^1,\dots,x^4,y^1,y^2,y^3\in \R$ and $\ui, \uj,\uk$ are the familiar units of the skew field of quaternions, which are known to obey the relations
\begin{align*}
\ui^2=\uj^2=\uk^2=-1\,, && \ui \uj=-\uj \ui= \uk.
\end{align*}
Note that each group $N_i$ is diffeomorphic to $\R^8$, and there are global coordinates 
\begin{align*}
N_1=H_1(2)_{x^1,\dots,x^4,y^1}\times \R^3_{y^2,y^3,x^5}\,, && 
N_2= H_2(1)_{x^1,\dots,x^4,y^2,y^3}\times\R^2_{y^1,x^5}\,,\\ N_3=H_3(1)_{x^1,\dots,x^4,y^1,y^2,y^3}\times\R_{x^5}\,.
\end{align*}

The Lie algebras of the $N_i$ can be characterized in terms of left-invariant frames $\{e_1,\dots,e_{8}\}$ satisfying the following structure equations: 
\begin{enumerate}

\vspace{0.3cm}
\item[$N_1$: ] $[e_1,e_2]=-[e_3,e_4]=e_5$, and all other brackets vanish; 

\vspace{0.4cm}
\item[$N_2$: ] $[e_1,e_3]=[e_2,e_4]=e_6,\,\, [e_1,e_4]=-[e_2,e_3]=e_7$, and all other brackets vanish; 

\vspace{0.4cm}
\item[$N_3$: ] $[e_1,e_2]=-[e_3,e_4]=e_5, \,\, [e_1,e_3]=[e_2,e_4]=e_6,\,\, [e_1,e_4]=-[e_2,e_3]=e_7$, and all other brackets vanish. 
\end{enumerate}
 
In each case, using the frame $\{e_1,\dots,e_8\}$ we can define the left-invariant HKT structure as consisting of the standard metric
$$
g=\sum_{r=1}^8 e^{r}\otimes e^r
$$
and the three complex structures  $(J_1,J_2,J_3)$ defined by 
$$
J_r(e_1)=e_{r+1}\,,\quad J_r(e_5)=e_{r+5}\,,\quad r=1,2,3\,.  
$$

Let us fix co-compact lattices

\begin{enumerate}

\vspace{0.3cm}
\item[] $\Gamma_1=\mathbb{Z}^3\times \left\{ \begin{pmatrix}
1 & a & c\\
0 & 1 & b^t\\
0 & 0 & 1
\end{pmatrix} \mid a,b\in \Z^2,\, c\in \Z \right\}\subset N_1$\,;

\vspace{0.4cm}
\item[] $\Gamma_2=\mathbb{Z}^2\times \left\{ \begin{pmatrix}
1 & z & u\\
0 & 1 & w\\
0 & 0 & 1
\end{pmatrix} \mid u,z,w\in \Z+\ui\Z \right\}\subset N_2$\,; 

\vspace{0.4cm}
\item[] $\Gamma_3=\mathbb{Z}\times \left\{ \begin{pmatrix}
1 & q & h-\frac{1}{2}q\bar{q}\\
0 & 1& -\bar q\\
0 & 0 & 1
\end{pmatrix} \mid q\in \Z+\ui\Z+\uj\Z+\uk\Z\,,\quad h\in \ui\Z+\uj\Z+\uk\Z \right\}\subset N_3$\,. 
\end{enumerate}

For $r=1,2,3$ we denote by $M_r=\Gamma_r\backslash N_r$ the compact nilmanifold obtained by quotienting $N_r$ by $\Gamma_r$. The 
left-invariant quadruple $(J_1,J_2,J_3,g)$ on $N_r$ induces an HKT structure on $M_r$.  Let $\{Z_1,\dots,Z_4\}$ indicate the left-invariant 
$(1,0)$-frame $Z_r=e_{2r-1}-\ui J_1(e_{2r-1})$, $r=1,\dots,4$, and denote by $\{\zeta^1,\dots,\zeta^4\}$  the dual $(1,0)$-coframe. Taking in account
$$
\partial_{J_2}=-J_2^{-1}\bar \partial J_2\,,
$$
we deduce the following identity, holding for every smooth real map $\phi$ on $M_r$:
\begin{equation*}
\begin{split}
\partial \partial_{J_2} \phi=&\partial J_2 \bar{\partial}\phi=\partial J_2 \left(\bar{Z}_1(\phi) \bar{\zeta}^1+\bar{Z}_2(\phi)\bar{\zeta}^2+\bar{Z}_3(\phi)\bar{\zeta}^3+\bar{Z}_4(\phi)\bar{\zeta}^4\right)\\
=&\partial \left(\bar{Z}_1(\phi) \zeta^2-\bar{Z}_2(\phi)\zeta^1+\bar{Z}_3(\phi)\zeta^4-\bar{Z}_4(\phi)\zeta^3\right)\\
=&\left(Z_1\bar{Z}_1(\phi)+Z_2\bar{Z}_2(\phi)\right)\zeta^{12}
+
\left(Z_3\bar{Z}_2(\phi)-Z_1\bar{Z}_4(\phi)\right)\zeta^{13}+\left(Z_4\bar{Z}_2(\phi)+Z_1\bar{Z}_3(\phi)\right)\zeta^{14}\\
&-
\left(Z_3\bar{Z}_1(\phi)+Z_2\bar{Z}_4(\phi)\right)\zeta^{23}+\left(Z_2\bar{Z}_3(\phi)-Z_4\bar{Z}_1(\phi)\right)\zeta^{24}+\left(Z_3\bar{Z}_3(\phi)+Z_4\bar{Z}_4(\phi)\right)\zeta^{34}\,.
\end{split}
\end{equation*}
Since
\begin{equation*}
\Omega=2(\zeta^{12}+\zeta^{34})\,,
\end{equation*}
it follows that
\begin{multline*}
(\Omega+\partial \partial_{J_2}\phi)^2= 2\left(Z_1\bar{Z}_1(\phi)+Z_2\bar{Z}_2(\phi)+2\right)\left(Z_3\bar{Z}_3(\phi)+Z_4\bar{Z}_4(\phi)+2\right)\zeta^{1234}\\
-2\left(Z_3\bar{Z}_2(\phi)-Z_1\bar{Z}_4(\phi)\right)\left(Z_2\bar{Z}_3(\phi)-Z_4\bar{Z}_1(\phi)\right)\zeta^{1234}\\
-2\left(Z_4\bar{Z}_2(\phi)+Z_1\bar{Z}_3(\phi)\right)\left(Z_3\bar{Z}_1(\phi)+Z_2\bar{Z}_4(\phi)\right)\zeta^{1234}\,,
\end{multline*}
in other words
\begin{multline}\label{general}
(\Omega+\partial \partial_{J_2}\phi)^2= 2\Bigl( \left(Z_1\bar{Z}_1(\phi)+Z_2\bar{Z}_2(\phi)+2\right)\left(Z_3\bar{Z}_3(\phi)+Z_4\bar{Z}_4(\phi)+2\right)\\
-\left(Z_3\bar{Z}_2(\phi)-Z_1\bar{Z}_4(\phi)\right)\left(Z_2\bar{Z}_3(\phi)-Z_4\bar{Z}_1(\phi)\right)\\
-\left(Z_4\bar{Z}_2(\phi)+Z_1\bar{Z}_3(\phi)\right)\left(Z_3\bar{Z}_1(\phi)+Z_2\bar{Z}_4(\phi)\right)\Bigr )\zeta^{1234}\,.
\end{multline}

Furthermore, every manifold $M_i$ is naturally a principal $T^3$-bundle over $T^5$ with projection 
$$
\pi\colon M_i\to T^{5}_{x^1\dots x^5}\,.
$$
A smooth function on $M_i$ is invariant under the action of the principal fibre $T^3$ if and only if it depends only on the five coordinates $\{x^1,\dots,x^5\}$. 
What is more, $T^3$-invariant functions on $M_i$ are naturally identified with  functions on $T^5$. As mentioned in the introduction, for a $T^3$-invariant real map  $F$ condition \eqref{normalization} becomes \eqref{normalization1}. 
 Further assuming that the HKT potential $\phi$ is $T^3$-invariant, equation \eqref{eq_quaternionic_Calabi_conjecture} can be written as \eqref{eq_equation} on $T^5$. 

\begin{oss}
The Lie algebras of the $ 2 $-step nilpotent Lie groups $ N_i $ all have $ 4 $-dimensional centre $ \mathfrak{z}=\{ e_5,e_6,e_7,e_8 \}$. 
Therefore the nilmanifolds $ M_i $ can be regarded in a natural way as principal $ T^4 $-bundles over a torus $ T^4 $ if we project onto the first four coordinates $ \{x^1,\dots,x^4\} $. From this point of view, requiring all data to be invariant under the action of the fibre $ T^4$ implies that the resulting equation can be written as the following Poisson equation on the base $ T^4 $
\[
\Delta \phi=\phi_{11}+\phi_{22}+\phi_{33}+\phi_{44}=\mathrm{e}^F-1\,.
\]
And this can be solved using standard techniques.
\end{oss}

\bigskip 
From this point on we shall focus on equation \eqref{eq_equation}. In order to simplify the notation let us set  
$$
A=\phi_{11}+\phi_{22}+\phi_{33}+\phi_{44}+1\,,\quad  B=\phi_{55}+1\,.
$$

\begin{lem}
If $ \phi\in C^2(T^5) $ is a solution to \eqref{eq_equation}, then $ A>0,B>0 $ and
\begin{equation}
\label{eq_Deltaphi+2>0}
0<2\mathrm{e}^{F/2}\leq \Delta \phi +2\,.
\end{equation}
\end{lem}
\begin{proof}
From equation \eqref{eq_equation} we infer $ AB\geq \mathrm{e}^F>0$. Hence $ A $ and $ B $ have the same sign. At a point $p_0$ where $ \phi $ attains its minimum we must have $ \phi_{55}(p_0)\geq 0$. This implies $B>0 $ and then $A>0$. Finally, by using $ A^2+B^2\geq 2AB $ we obtain 
\[
(\Delta \phi+2)^2=(A+B)^2\geq 4AB\geq 4\mathrm{e}^F>0\,.
\]
Taking the square root produces \eqref{eq_Deltaphi+2>0}.
\end{proof}

\begin{prop}
\label{prop_ellipticity}
Equation \eqref{eq_equation} is elliptic. More precisely, if $ \phi\in C^2(T^5) $ denotes a solution to \eqref{eq_equation} then
\begin{equation}
\label{eq_ellipticity}
A\xi_5^2+B(\xi_1^2+\xi_2^2+\xi_3^2+\xi_4^2)-2\sum_{i=1}^4\phi_{i5}\xi_i\xi_5 \geq \lambda(\phi) (\xi_1^2+\xi_2^2+\xi_3^2+\xi_4^2+\xi_5^2)
\end{equation}
for every $ (\xi_1,\xi_2,\xi_3,\xi_4,\xi_5)\in \R^5 $, where
\begin{equation*}
\lambda(\phi)=\frac{1}{2}\left(A+B- \sqrt{(A+B)^2-4\mathrm{e}^F}\right).
\end{equation*}
\end{prop}
\begin{proof}
The principal symbol of  the linearized equation at a solution $ \phi $ equals
\begin{equation*}
A\xi_5^2+B(\xi_1^2+\xi_2^2+\xi_3^2+\xi_4^2)-2\phi_{15}\xi_1\xi_5-2\phi_{25}\xi_2\xi_5-2\phi_{35}\xi_3\xi_5-2\phi_{45}\xi_4\xi_5
\end{equation*}
and the corresponding matrix is
\[
P(\phi)=\begin{pmatrix}
B & 0 & 0 & 0 & -\phi_{15} \\
0 & B & 0 & 0 & -\phi_{25}  \\
0 & 0 & B & 0 & -\phi_{35}\\
0 & 0 & 0 & B & -\phi_{45} \\
-\phi_{15} & -\phi_{25} & -\phi_{35} & -\phi_{45} & A\\
\end{pmatrix}\,.
\]
Since, by \eqref{eq_equation}, 
\[
\begin{split}
\det(P(\phi)- \lambda I)&=(B-\lambda)^3\left( (A-\lambda)(B-\lambda)-(\phi_{15}^2+\phi_{25}^2+\phi_{35}^2+\phi_{45}^2) \right)\\
&=(B-\lambda)^3\left( \lambda^2 -(A+B)\lambda +\mathrm{e}^F \right)\,,
\end{split}
\]
the eigenvalues are $ \lambda= B $ and
\[
\lambda_\pm=\frac{1}{2}\left(A+B\pm \sqrt{(A+B)^2-4\mathrm{e}^F}\right)\,.
\]
Now $ (A+B)^2-4\mathrm{e}^F\geq(A-B)^2=((A+B)-2A)^2=((A+B)-2B)^2 $, so that
\[
0<\lambda_-\leq B\leq \lambda_+\,.
\]
This proves the claim.
\end{proof}
  
\section{$ C^0 $-estimate} \label{C^0}

Although the a priori $ C^0 $-estimate for equation \eqref{eq_equation} can be deduced from the $C^0$-estimate of the quaternionic Monge-Amp\`ere equation, as  shown in \cite{Alesker-Shelukhin (2017),Alesker-Verbitsky (2010),Sroka}, we shall prove this fact using an argument that is specific to our setup. 

\medskip 

Call $ B_R(x_0) $ the open ball in $ \R^N $ centred at $ x_0 $ and of radius $ R>0 $. We need to recall the following results:
\begin{teor}[Weak Harnack Estimate, Theorem 8.18 in \cite{Gilbarg-Trudinger (1983)}]
\label{teor_weak_Harnack}
Consider $ 1\leq p <N/(N-2) $, and $ q>N $, where $ N>2 $ is an integer. For every $ R>0 $ there exists a positive constant $ C=C(N,R,p,q) $ such that
\begin{equation*}
r^{-N/p}\|u\|_{L^p(B_{2r}(x_0))}\leq C\left( \inf_{x\in B_r(x_0)} u(x)+r^{2-2N/q}\|f\|_{L^{q/2}(B_{R}(x_0))} \right),
\end{equation*}
for any $ x_0\in \R^N $, $ 0<r<R/4 $, $ f\in C^0(\R^N) $, and any $ u \in C^2(\R^N) $ that is non-negative on $ B_R(x_0) $ and such that $ \Delta u(x)\leq f(x) $ for all $ x\in B_R(x_0) $.
\end{teor}

\begin{teor}[Sz\'{e}kelyhidi, \cite{Szekelyhidi}]
\label{teor_Szekelyhidi}
Consider a map $ u\in C^2(\R^N) $ and assume there exist a point $ x_0\in \R^N $ and numbers $ R>0 $ and $ \epsilon>0 $ such that $ \min_{\abs{x-x_0}\leq R} u(x)=u(x_0) $, and
\[
u(x_0)+2R\epsilon \leq \min_{\abs{x-x_0}=R} u(x)\,.
\]
Then 
\begin{equation*}
\epsilon^N\leq \frac{2^N}{\abs{B_R(0)}}\int_{\Gamma_\epsilon} \det(D^2u)\,,
\end{equation*}
where
\[
\Gamma_\epsilon =\left\{ x\in B_R(x_0)\mid  u(y)\geq u(x)+\nabla u(x)\cdot (y-x),\, \forall y\in B_R(x_0),\, \abs{\nabla u(x)}<\frac{\epsilon}{2} \right\}.
\]
\end{teor}

Now, let us identify functions on $ T^5 $ with functions $ \phi \colon \R^5 \to \R $ that are periodic in each variable. Denote by $ C^n(T^5) $ the Banach space of functions $ \phi \colon T^5 \to \R $ with $ C^n $-norm
\[
\norma{\phi}_{C^n}=\max_{\abs{I}\leq n} \sup_{x\in \R^5}\abs{\partial^I \phi(x)}
\]
where $ I=\{i_1,\dots,i_5\} $. We are adopting the multi-index notation $ \partial^I=\partial_1^{i_1} \partial_2^{i_2}\partial_3^{i_3}\partial_4^{i_4}\partial_5^{i_5} $ with $ \abs{I}=i_1+i_2+i_3+i_4+i_5. $ For $ \alpha \in (0,1) $ we also consider the Banach space $ C^{n,\alpha}(T^5) $ of functions $ \phi\in C^n(T^5) $ with H\"{o}lder-continuous derivatives of order $ n $: 
\[
\norma{\phi}_{C^{n,\alpha}}=\max\{ \norma{\phi}_{C^n}, \abs{\phi}_{C^{n,\alpha}} \} <\infty\,,
\]
where
\[
\abs{\phi}_{C^{n,\alpha}}=\max_{\abs{I}=n}\sup_{x\in \R^3} \sup_{0<\abs{h}\leq 1}	\frac{\abs{\partial^I\phi(x+h)-\partial^I\phi(x)}}{\abs{h}^\alpha}\,.
\]

Set
\[
C_*^k(T^5)=\left\{ \phi \in C^k(T^5) \mid \int_{K} \phi=0  \right\}
\]
where
\[
K=\left[ -\frac{1}{2},\frac{1}{2}\right]^5. 
\]

\begin{teor}
Assume that $ F\in C^0(T^5) $ satisfies \eqref{normalization1}. Let $ \phi \in C^2_*(T^5) $ be a solution to \eqref{eq_equation}. Then there is a positive constant $ C_0 $, depending on $ \norma{F}_{C^0} $ only, such that
\begin{equation}
\label{eq_C0-estimate}
\norma{\phi}_{C^0}\leq C_0\,.
\end{equation}
\end{teor}
\begin{proof}
Let $ x_0\in \R^5 $ be a point where $ \phi $ attains its minimum on $ K $. Fix $ \epsilon>0 $ and define
\begin{equation}
\label{eq_02}
u(x)=\phi(x)-\max_K\phi+4\epsilon \abs{x-x_0}^2\,.
\end{equation}
Then
\begin{equation*}
u(x_0)+\epsilon=\phi(x_0)-\max_K\phi+\epsilon\leq \min_{\abs{x-x_0}=1/2}\phi(x)-\max_K \phi+\epsilon=\min_{\abs{x-x_0}=1/2}u(x)
\end{equation*}
and by Theorem \ref{teor_Szekelyhidi}, with $ R=1/2 $, we have 
\begin{equation}
\label{eq_epsilon}
\epsilon^5\leq\frac{2^5}{\abs{B_{1/2}(0)}}\int_{\Gamma_\epsilon}\det(D^2u)\,.
\end{equation}
Differentiating \eqref{eq_02} twice gives $ D^2u=D^2\phi+8\epsilon I $. Hence we may rewrite equation \eqref{eq_equation} as 
\begin{equation}
\label{eq_03}
(u_{11}+u_{22}+u_{33}+u_{44}-32\epsilon+1)(u_{55}-8\epsilon+1)-u_{15}^2-u_{25}^2-u_{35}^2-u_{45}^2=\mathrm{e}^F\,.
\end{equation}
Now, on $ \Gamma_\epsilon $ the function $ u $ is convex, therefore the Hessian matrix $ D^2u(x) $ is non-negative for all $ x\in \Gamma_\epsilon $. In particular $ u_{ii}(x)\geq 0 $ for all $ i=1,\dots,5 $ and every $ x\in \Gamma_\epsilon $. In addition, 
\begin{equation}
\label{eq_03b}
u_{ii}(x)u_{55}(x)-u_{i5}^2(x)\geq 0, \qquad \text{for all } i=1,\dots,5, \, \text{ and every } x\in \Gamma_\epsilon\,.
\end{equation}

Set $ \epsilon=\epsilon_0=1/48 $, so that from  \eqref{eq_03b} and \eqref{eq_03} we obtain, for every $ x\in \Gamma_{\epsilon_0} $,
\begin{align*}
\frac{\Delta u(x)}{5}&\leq \frac{5}{6}(u_{11}(x)+u_{22}(x)+u_{33}(x)+u_{44}(x))+\frac{1}{3}u_{55}(x)\\
&\leq \left(u_{11}(x)+u_{22}(x)+u_{33}(x)+u_{44}(x)+\frac{1}{3}\right)\left(u_{55}(x)+\frac{5}{6}\right)- \sum_{i=1}^{4}u_{i5}^2(x)-\frac{5}{18}\\
&=\mathrm{e}^{F(x)}-\frac{5}{18}\leq \mathrm{e}^{\max_K F}.
\end{align*}
Using again the fact that $ D^2u $ is non-negative on $ \Gamma_\epsilon $, the arithmetic-geometric mean inequality forces 
\begin{equation}
\label{eq_03c}
\det(D^2u(x))\leq \left( \frac{\Delta u(x)}{5} \right)^5\leq \mathrm{e}^{5\max_K F},\qquad \text{for every } x\in \Gamma_{\epsilon_0}\,.
\end{equation}
At last, \eqref{eq_epsilon} and \eqref{eq_03c} imply 
\begin{equation*}
\left(\frac{1}{48}\right)^5\frac{\abs{B_{1/2}(0)}}{2^5}\leq \int_{\Gamma_{\epsilon_0}} \det (D^2u)\leq \mathrm{e}^{5\max_KF} \mathrm{meas}(\Gamma_{\epsilon_0})\,,
\end{equation*}
i.e.
\begin{equation}
\label{eq_meas_Gamma}
\mathrm{meas}(\Gamma_{\epsilon_0})\geq \abs{B_{1/2}(0)}\left( \frac{\mathrm{e}^{-\max_KF}}{96} \right)^5=:C\,.
\end{equation}

Now observe that 
\[
u(x)\leq u(x_0)-\nabla u(x)\cdot (x_0-x)\leq u(x_0)+\frac{\epsilon_0}{4}\,,\qquad \text{for every } x\in \Gamma_{\epsilon_0}\,,
\]
that is
\[
\phi(x)-\max_K \phi+4\epsilon_0\abs{x-x_0}^2\leq \phi(x_0)-\max_K \phi+\frac{\epsilon_0}{4}=\min_K \phi- \max_K \phi +\frac{\epsilon_0}{4}\,, \qquad \text{for every } x\in \Gamma_{\epsilon_0}\,.
\]
This implies
\[
\max_K \phi- \min_K \phi \leq \max_K \phi - \phi(x)+ 1\,, \qquad \text{for every } x\in \Gamma_{\epsilon_0}\,.
\]
It follows that for every $ p\geq 1 $
\begin{equation*}
\left( \max_K \phi- \min_K\phi \right)\left( \mathrm{meas}(\Gamma_{\epsilon_0}) \right)^{1/p}\leq \left( \int_{\Gamma_{\epsilon_0}}\left( \max_K \phi-\phi+1 \right)^p \right)^{1/p}=\left \lVert \max_K \phi-\phi+1 \right \rVert_{L^p(\Gamma_{\epsilon_0})}\,,
\end{equation*}
and since $ \Gamma_{\epsilon_0}\subseteq B_{1/2}(x_0)\subseteq K+x_0 $, we have
\[
\left \lVert \max_K \phi-\phi+1 \right \rVert_{L^p(\Gamma_{\epsilon_0})}\leq \left \lVert \max_K \phi-\phi+1 \right \rVert_{L^p(K)}\,.
\]
Therefore, since $ \int_K\phi=0 $, we have $ \norma{\phi}_{C^0}\leq \max_K \phi-\min_K \phi $. Then \eqref{eq_meas_Gamma} implies 
\begin{equation}
\label{eq_phi_C0_Lp}
\norma{\phi}_{C^0}\leq \max_K \phi - \min_K \phi\leq C^{-1/p}\left(\left \lVert \max_K \phi-\phi\right \rVert_{L^p(K)} +1\right)\,, \qquad \forall p\geq 1\,.
\end{equation}
By \eqref{eq_Deltaphi+2>0} we see that $ \Delta(\max_K\phi-\phi)\leq 2 $, and since $ \max_K \phi-\phi \geq 0 $ we can apply Theorem \ref{teor_weak_Harnack} with $ \max_K \phi-\phi $ in place of $ u $, $ N=5 $, $ p=4/3 $, $ q=6 $, $ x_0\in K $ such that $ \phi(x_0)=\max_K \phi $, $ r=1/2 $ and $ R=3 $. This eventually shows there exists a positive constant $ C' $ satisfying
\begin{equation}
\label{eq_phi_Harnack}
\left \lVert \max_K \phi-\phi \right \rVert_{L^{4/3}(K)}\leq C'\left( \inf_K\left( \max_K \phi -\phi \right)+\norma{2}_{L^3(K)} \right)=2C'\,.
\end{equation}
Estimate \eqref{eq_C0-estimate} now follows from \eqref{eq_phi_C0_Lp} with $ p=4/3 $ and \eqref{eq_phi_Harnack}.
\end{proof}

\section{$ C^0 $-estimate for the Laplacian} \label{Laplacian}
In this section we shall prove a $ C^0 $-estimate for the Laplacian of $\varphi$. The technique we employ is an adaptation of that found in \cite{Buzano-Fino-Vezzoni (2015)}.

\begin{lem}\label{new_Luigi}
Let $\varphi$ be a $C^2$ function on the $n$-torus $T^n$, fix $\mu\in\R$ and pick a point $p_0$ where $ \Phi=(\Delta \phi +2)\mathrm{e}^{-\mu \phi}$ attains its maximum value. Define
$$
\eta_{ij}=\mu (\Delta \phi+2)(\phi_{ij}+\mu \phi_i\phi_j)-\Delta \phi_{ij}\,,\quad i,j=1,\dots,n\,.
$$
Then
$$
\eta_{ii}(p_0)\geq0	\,,\,\, \mbox{ and }\, \sqrt{\eta_{ii}\eta_{jj}}\geq |\eta_{ij}| \mbox{ at  }p_0\,,
$$
for every $i,j=1,\dots, n$.
\end{lem}
\begin{proof}
We begin by recalling the standard formulas
$$
\nabla \Phi=\mathrm{e}^{-\mu \phi} \left( \nabla \Delta \phi-\mu (\Delta \phi+2)\nabla \phi \right)
$$
and 
\[
\begin{split}
(\nabla \otimes \nabla) \Phi=&-\mu \mathrm{e}^{-\mu \phi} \bigl ( \nabla \phi \otimes \nabla \Delta \phi+\nabla \Delta \phi\otimes \nabla \phi \bigr )+\mu^2\mathrm{e}^{-\mu \phi}\bigl ( (\Delta \phi +2)\nabla \phi\otimes \nabla \phi \bigr )\\
&+ \mathrm{e}^{-\mu \phi}\bigl ( (\nabla \otimes \nabla) \Delta \phi- \mu (\Delta \phi + 2)(\nabla \otimes \nabla) \phi \bigr )\,.
\end{split}
\]
Since 
$$
 \nabla \Phi=0\,,\quad  (\nabla \otimes \nabla)\Phi\leq 0\quad \mbox{ at } p_0\,,
$$
we infer 
\begin{equation}\label{eq_nablaPhi=0}
\nabla \Delta \phi=\mu (\Delta \phi+2)\nabla \phi\quad \mbox{ at }p_0\,
\end{equation}
and
\[
(\nabla \otimes \nabla) \Delta\phi\leq \mu(\Delta \phi +2)((\nabla \otimes \nabla)\phi+\mu \nabla \phi \otimes \nabla \phi) \quad 
\mbox{ at }p_0\,.
\]
In particular
\begin{multline*}
\left( \mu(\Delta \phi+2)(\phi_{ij}+\mu \phi_i\phi_j)- \Delta \phi_{ij} \right)^2\\
\leq\left( \mu (\Delta \phi+2)(\phi_{ii}+\mu \phi_i^2)-\Delta \phi_{ii} \right)\left( \mu(\Delta \phi +2)(\phi_{jj}+\mu \phi_j^2)-\Delta \phi_{jj} \right)
\end{multline*}
at $p_0$, for every $ 1\leq i,j\leq n $, and also
\begin{equation*}
\mu (\Delta \phi+2)(\phi_{ii}+\mu \phi_i\phi_i)-\Delta \phi_{ii}\geq 0 \quad \mbox{ at } p_0\,,\quad i=1,\dots, n.
\end{equation*}
Hence the claim follows. 
\end{proof}

\begin{prop}
\label{prop_7}
Let $ F\in C^2(T^5) $ satisfy \eqref{normalization1}. There exists a positive constant $ C_1 $, depending on $ \norma{F}_{C^2} $ only, such that 
\begin{equation}
\label{eq_estimate_laplacian.}
\norma{\Delta \phi}_{C^0}\leq C_1(1+\norma{\phi}_{C^1})\,
\end{equation}
for  any solution $\phi\in C^4_*(T^5) $ to \eqref{eq_equation}.
\end{prop}
\begin{proof} 
For starters, 
\begin{equation}
\label{eq_Deltae^F}
\begin{split}
\Delta \mathrm{e}^F=\Delta AB+A\Delta B+2\nabla A\cdot \nabla B-2\sum_{i=1}^{4} 
\left(\abs{\nabla \phi_{i5}}^2+\phi_{i5} \Delta \phi_{i5}\right).
\end{split}
\end{equation}
Let $p_0$ and $\eta_{ij}$ be as in Lemma \ref{new_Luigi} with 
$$
\mu=\frac{\epsilon}{\max(\Delta \phi+2)}
$$
and $\epsilon\in(0,1) $ to be determined later. Then by using \eqref{eq_ellipticity} with 
$$
\xi_i =\mathrm{sgn}(\phi_{i5})\sqrt{\eta_{ii}},\,\, i=1,\dots,4\,,\quad 
\xi_5=\sqrt{\eta_{55}}\,,
$$
we find
\begin{equation*}
\mu(\Delta \phi +2)\left( A(\phi_{55}+\mu \phi_5^2)+B\sum_{i=1}^{4} (\phi_{ii}+\mu \phi_i^2) \right)
-A\underbrace{\Delta \phi_{55}}_{\Delta B}-B\underbrace{\sum_{i=1}^{4}
\Delta \phi_{ii}}_{\Delta A}-2\sum_{i=1}^{4}\phi_{i5}\xi_i\xi_5\geq 0\,.
\end{equation*}
at $ p_0 $. Lemma \ref{new_Luigi} now implies 
\[
\phi_{i5}\xi_i\xi_5=\abs{\phi_{i5}}\sqrt{\eta_{ii}}\sqrt{\eta_{55}}\geq \phi_{i5}\eta_{i5}\,,\,\,\mbox{ at }p_0\,,
\]
i.e. 
\[
\phi_{i5}\xi_i\xi_5\geq  \phi_{i5}\left(\mu(\Delta \phi+2)(\phi_{i5}+\mu \phi_i\phi_5)- \Delta \phi_{i5}\right)\,\, \mbox{at $p_0$}\,.
\]
Therefore we obtain
\begin{multline*}
\mu(\Delta \phi +2)\left( A(\phi_{55}+\mu \phi_5^2)+B\sum_{i=1}^{4} (\phi_{ii}+\mu \phi_i^2) \right)-2\sum_{i=1}^{4}\phi_{i5}\left(\mu(\Delta \phi+2)(\phi_{i5}+\mu \phi_i\phi_5)\right)\\\geq 
A\Delta B+B\Delta A-2\sum_{i=1}^4 \phi_{i5}\Delta \phi_{i5}\,, \quad \mbox{ at }p_0\,.
\end{multline*}

By \eqref{eq_Deltae^F}, and the definition of $ A,B $, at the point $ p_0 $ we have
\begin{align*}
\Delta\mathrm{e}^F\leq& \mu(\Delta \phi +2)\left( A(B-1)+B(A-1) \right)+2\nabla A\cdot \nabla B\\
&+\mu^2(\Delta \phi+2)\left( A\phi_{5}^2+B\sum_{i=1}^{4}\phi_i^2\right)
-2\mu(\Delta \phi+2)\sum_{i=1}^{4}\left(\phi_{i5}^2+\mu\phi_{i5} \phi_i\phi_5\right)\\
=& 2\mu(\Delta \phi+2)\left(AB-\sum_{i=1}^{4}\phi_{i5}^2\right)-\mu(\Delta \phi+2)(A+B)+2\nabla A\cdot \nabla B\\
&+\mu^2(\Delta \phi+2)\left( A\phi_{5}^2+B(\phi_{1}^2+\phi_2^2+\phi_3^2+\phi_4^2)-2\sum_{i=1}^{4}\phi_{i5} \phi_i\phi_5\right)\\
\leq & 2\mu(\Delta \phi+2)\mathrm{e}^F-\mu(\Delta \phi+2)^2+2\nabla A\cdot \nabla B
+2\mu^2(\Delta \phi+2)\left( A\phi_{5}^2+B(\phi_{1}^2+\phi_2^2+\phi_3^2+\phi_4^2)\right).
\end{align*}
Observe that in the last inequality we used \eqref{eq_ellipticity} with $ \xi_i=\phi_i(p_0) $ for $ i=1,\dots,4 $ and $\xi_5=-\phi_5(p_0) $. 

By  \eqref{eq_nablaPhi=0} we then have
\[
\mu^2(\Delta \phi+2)^2\abs{\nabla \phi}^2=\abs{\nabla \Delta \phi}^2=\abs{\nabla(A+B)}^2=\abs{\nabla A}^2+\abs{\nabla B}^2+ 2\nabla A\cdot \nabla B\geq 2\nabla A\cdot \nabla B\,, \quad \mbox{ at }p_0\,,
\]
and with the help of 
\[
A\phi_{5}^2+B(\phi_{1}^2+\phi_2^2+\phi_3^2+\phi_4^2)\leq A\abs{\nabla \phi}^2+B\abs{\nabla \phi}^2=(\Delta \phi+2)\abs{\nabla \phi}^2
\]
we deduce
\begin{equation}
\label{eq_05}
\mu(\Delta \phi(p_0)+2)^2\leq -\Delta \mathrm{e}^{F}(p_0)+ 2\mu(\Delta \phi(p_0)+2)\mathrm{e}^{F(p_0)}+3\mu^2(\Delta \phi(p_0)+2)^2\abs{\nabla \phi(p_0)}^2.
\end{equation}
Let us set 
\begin{align*}
m=\Delta \phi(p_0)+2\,,\,\quad \phi_0=\phi(p_0)\,.
\end{align*}
Since $ p_0 $ is a maximum point for $ \Phi $, clearly
\[
\max \Phi=m\mathrm{e}^{-\mu \phi_0}\,.
\]
From \eqref{eq_05} we obtain
\begin{equation}
\label{eq_06}
\mu m^2\leq \left \lVert \Delta \mathrm{e}^F \right \rVert_{C^0}+2 \mu m \norma{\mathrm{e}^F}_{C^0}+ 3\mu^2m^2 \norma{\nabla \phi}^2_{C^0}.
\end{equation}
Now fix a point $p_1$ where $\Delta \phi+2$ reaches its maximum, and call $ \phi_1=\phi(p_1)$. Then
\begin{equation}
\label{eq_07}
m\leq \max(\Delta \phi +2)=\mathrm{e}^{\mu \phi_1} \Phi \leq m \mathrm{e}^{\mu(\phi_1-\phi_0)}\leq m 
\mathrm{e}^{2\mu\norma{\phi}_{C^0}}.
\end{equation}
By the definition of $ \mu $ and inequality \eqref{eq_Deltaphi+2>0} we have
\[
2\mu =\frac{2}{\max(\Delta \phi+2)}\,\epsilon\leq \frac{1}{\mathrm{e}^{\min(F/2)}}\,\epsilon\leq \mathrm{e}^{-\min(F/2)}\,,
\]
hence by \eqref{eq_07}
\begin{equation*}
\epsilon \exp\left( -\mathrm{e}^{-\min(F/2)}\norma{\phi}_{C^0} \right)\leq \epsilon \mathrm{e}^{-2\mu\norma{\phi}_{C^0}}=\mu \max(\Delta \phi+2)\mathrm{e}^{-2\mu\norma{\phi}_{C^0}}\leq \mu m
\end{equation*}
and also
\begin{equation*}
\exp\left( -\mathrm{e}^{-\min(F/2)}\norma{\phi}_{C^0} \right)\max(\Delta \phi+2)\leq \mathrm{e}^{-2\mu\norma{\phi}_{C^0}}\max(\Delta \phi+2)\leq m\,.
\end{equation*}
Next we multiply the last two inequalities and use \eqref{eq_06}, recalling that $ \mu m\leq \epsilon $, to the effect that
\begin{equation*}
\epsilon \exp\left( -2\mathrm{e}^{-\min(F/2)}\norma{\phi}_{C^0} \right)\max(\Delta \phi+2)\leq   \left \lVert \Delta \mathrm{e}^F \right \rVert_{C^0}+2 \epsilon \norma{\mathrm{e}^F}_{C^0}+ 3\epsilon^2 \norma{\nabla \phi}^2_{C^0}\,.
\end{equation*}
Put otherwise, 
\begin{equation*}
\norma{\Delta \phi}_{C^0}\leq \exp\left( 2\mathrm{e}^{-\min(F/2)}\norma{\phi}_{C^0} \right)\left(  \frac{1}{\epsilon} \left \lVert \Delta \mathrm{e}^F \right \rVert_{C^0}+2 \norma{\mathrm{e}^F}_{C^0}+ 3\epsilon \norma{\nabla \phi}^2_{C^0}  \right),
\end{equation*}
and by choosing 
\[
\epsilon=\frac{1}{1+\norma{\nabla \phi}_{C^0}}
\]
the claim is straighforward.
\end{proof}

The next theorem will provide us with an a priori $C^1$-estimate for $\phi$. Together with Proposition \ref{prop_7} it will give an a priori $C^0$-bound for $\Delta \varphi$.  

\begin{teor}
For all solutions $ \phi\in C^4_*(T^5) $ of equation \eqref{eq_equation} with $ F\in C^2(T^5) $ satisfying \eqref{normalization1} there exists a positive constant $ C_2 $, depending on $ \norma{F}_{C^2} $ only, such that
\begin{equation}
\label{eq_C1-estimate}
\norma{\phi}_{C^1}\leq C_2.
\end{equation}
\end{teor}
\begin{proof}
Fix $ 0<\alpha<1 $ and $ p=\frac{3}{1-\alpha}>3 $. Morrey's inequality says 
\[
\norma{\phi}_{C^{1,\alpha}}\leq C \norma{\phi}_{W^{2,p}}
\]
for some positive constant $ C $ depending only on $ \alpha $. 
Elliptic $ L^p $-estimates for the Laplacian also generate another constant $ C' $,  still depending on $ \alpha $ only, such that
\[
\norma{\phi}_{W^{2,p}}\leq C'\left( \norma{\phi}_{L^p}+\norma{\Delta u}_{L^p} \right)\,.
\]
If $ \phi \in C^2(T^5) $, the $ C^0 $-estimate \eqref{eq_C0-estimate} for $ \phi $ and  bound \eqref{eq_estimate_laplacian.} for $ \Delta \phi $ imply 
\[
\norma{\phi}_{L^p}+\norma{\Delta \phi}_{L^p}\leq \norma{\phi}_{C_0}+\norma{\Delta \phi}_{C^0}\leq C_0+C_1(1+\norma{\phi}_{C^1})\,.
\]

Using standard interpolation theory (see \cite[section 6.8]{Gilbarg-Trudinger (1983)}), for any $ \epsilon>0 $ there is a constant $ P_\epsilon>0 $ such that
\[
\norma{\phi}_{C^1}\leq P_\epsilon \norma{\phi}_{C^0}+\epsilon\norma{\phi}_{C^{1,\alpha}}, \qquad \text{for every } \phi\in C^{1,\alpha}(T^5)\,.
\]
Putting all this together, we obtain 
\[
\norma{\phi}_{C^1}\leq P_\epsilon C_0+\epsilon K_0\left( C_0+C_1(1+\norma{\phi}_{C^1}) \right)= P_\epsilon C_0+\epsilon K_0(C_0+C_1)+\epsilon K_0C_1\norma{\phi}_{C^1}\,, 
\]
for some positive constant $ K_0 $, again depending on $ \alpha $ only. 
This produces \eqref{eq_C1-estimate} once we choose
\[
\epsilon<\frac{1}{K_0C_1}\,. \qedhere
\]
\end{proof}

\begin{cor}\label{corr??}
Assume that $ F\in C^2(T^5) $ satisfies \eqref{normalization1} and let $ \phi\in C^4_*(T^5) $ be a solution to \eqref{eq_equation}. Then there exists a positive constant $ C_3 $, depending on $ \norma{F}_{C^2} $ only, such that
\begin{equation*}
\norma{\Delta \phi}_{C^0}\leq C_3\,.
\end{equation*}
\end{cor}

  
\section{$ C^{2,\alpha} $-estimate}\label{C^2,alpha}

The $C^{2,\alpha} $-estimate for our equation \eqref{eq_equation} can be deduced directly from the general result of Alesker which we state next. It holds for compact hypercomplex manifolds that are locally flat, in the sense that they are locally isomorphic to $\mathbb H^n$.  

\begin{teor}[Theorem 4.1 in \cite{Alesker (2013)}]\label{locflatth}
Let  $M$ be a $4n$-dimensional compact HKT manifold whose underlying hypercomplex structure is locally flat. Suppose $ \phi\in C^2(M) $ is a solution to the quaternionic Monge-Amp\`ere equation \eqref{eq_quaternionic_Calabi_conjecture}.  
Then
$$
\norma{\phi}_{C^{2,\alpha}}\leq C\,
$$
for some $ \alpha\in (0,1)$ and a positive constant $ C$, both depending on $M$, $\Omega,$  $\norma{F}_{C^2},$  $ \norma{\phi}_{C_0}$  and $\|\tilde \Delta \phi\|_{C^0} $,  where 
$$
\tilde \Delta \phi=	\frac{\partial \partial_{J_2}\phi\wedge \Omega^{n-1} }{\Omega^{n}}
$$
and $\Omega$ is the HKT form.
\end{teor}

The HKT structures we are considering on $M_r$ are flat for the Obata connection \cite[Proposition 6.1]{Dotti-Fino (2000)}. Hence the underlying hypercomplex structure is  locally flat. Moreover, for $T^3$-invariant functions the operator $\tilde \Delta$ acts as a multiple of the Laplace operator, hence Theorem \ref{locflatth} and Corollary \ref{corr??} imply 

\begin{prop}\label{prop??}
Assume $ F\in C^2(T^5) $ satisfies \eqref{normalization1}. For every solution $ \phi\in C^4_*(T^5) $ to equation \eqref{eq_equation} there exist $ \alpha \in (0,1) $ and a positive constant $ C_4 $, depending on $ \norma{F}_{C^2}, \norma{\phi}_{C^0} $ only, such that
\begin{equation*}
\norma{\phi}_{C^{2,\alpha}}\leq C_4\,.
\end{equation*}
\end{prop}

\section{Proof of Theorem \ref{main}} \label{Proof}

In this section we shall use the previously established a priori estimates in order to prove the following result. This will then imply Theorem \ref{main}.

\begin{teor}
Let $ F\in C^\infty(T^5)$ satisfy \eqref{normalization1}. Then equation \eqref{eq_equation} admits a solution $ \phi \in C^\infty_*(T^5) $.
\end{teor}
\begin{proof}
For $ t\in [0,1] $ we define
\[
F_t=\log(1-t+t\mathrm{e}^F)
\]
and set 
\[
S_t=\left\{ \phi \in C^\infty_*(T^5) \mid  (\phi_{11}+\phi_{22}+\phi_{33}+\phi_{44}+1)(\phi_{55}+1)-\phi_{15}^2-\phi_{25}^2-\phi_{35}^2-\phi_{45}^2=\mathrm{e}^{F_t} \right\},
\]
and $ S=\bigcup_{t\in [0,1]} S_t $. Clearly $ 0\in S_0 $, and $ S_1 $ is the set of smooth solutions of \eqref{eq_equation}. We thus need to show that $ S_1\neq \emptyset $. For any $ t\in [0,1]$ the map  $F_t$ satisfies \eqref{normalization1} and 
\[
\max_{t\in [0,1]}\norma{F_t}_{C^2}<\infty\,.
\]
 Proposition \ref{prop??} therefore implies there exists $ \alpha\in (0,1) $ such that
\begin{equation}
\label{eq_phi_k_bounded}
\sup_{\phi\in S} \norma{\phi}_{C^{2,\alpha}}<\infty\,.
\end{equation}
Let 
\[
\tau=\sup\{ t\in [0,1] \mid S_t\neq \emptyset \}\,.
\]
We claim that $S_{\tau}\neq \emptyset $ and $ \tau=1. $

\begin{enumerate}
\item[$S_{\tau}\neq \emptyset $.] Let $ \{t_k\}\subseteq [0,1] $ be an increasing sequence converging to $\tau$, and for any $k\in \mathbb N $ we fix $\phi_k\in S_{t_k}$. Condition \eqref{eq_phi_k_bounded} implies that $\{\phi_k\}$ is a sequence in  $C^{2,\alpha}_*(T^5) $, so by the Ascoli-Arzel\`{a} Theorem there exists a subsequence $\{\phi_{k_j}\}$ converging to some $ \psi $ in $ C^{2,\alpha/2}_*(T^5) $. The function $ \psi $ satisfies
\[
(\psi_{11}+\psi_{22}+\psi_{33}+\psi_{44}+1)(\psi_{55}+1)-\psi_{15}^2-\psi_{25}^2-\psi_{35}^2-\psi_{45}^2=\mathrm{e}^{F_\tau}\,.
\]
In view of Proposition \ref{prop_ellipticity}, equation \eqref{eq_equation} is elliptic,  and elliptic regularity (see e.g. \cite[Theorem 4.8, Chapter 14]{Taylor}) implies that $\psi$ is in fact $C^{\infty}$. Therefore $S_\tau \neq \emptyset $, as required.

\vspace{0.2cm}
\item[$ \tau=1$.] Assume, by contradiction, that $ \tau<1 $, and consider the non-linear operator 
$$
T\colon C^{2,\alpha}_*(T^5)\times[0,1] \to C^{0,\alpha}_*(T^5) 
$$ 
defined by 
\[
T(\phi,t)=(\phi_{11}+\phi_{22}+\phi_{33}+\phi_{44}+1)(\phi_{55}+1)-\phi_{15}^2-\phi_{25}^2-\phi_{35}^2-\phi_{45}^2-\mathrm{e}^{F_t}\,.
\]
Since $S_\tau\neq\emptyset $, there exists $\psi \in C^{\infty}_{*}(T^5)$ such that $ T(\psi,\tau)=0 $. Let $L\colon C^{2,\alpha}_*(T^5)\to C^{0,\alpha}_*(T^5) $ be the first variation of $T$ with respect to the first variable. Then 
\[
Lu=Au_{55}+B(u_{11}+u_{22}+u_{33}+u_{44})-2C_1u_{15}-2C_2u_{25}-2C_3u_{35}-2C_4u_{45}
\]
where 
\begin{align*}
A=(\psi_{11}+\psi_{22}+\psi_{33}+\psi_{44}+1)\,, && B=(\psi_{55}+1)\,, && C_i=\psi_{i5}\,,
\end{align*}
which implies that $L$ is elliptic since $ \psi \in S_\tau $. 
The strong maximum principle guarantees $L$ is injective because $ L\phi=0 $ forces $ \phi$  to be constant. Furthermore, ellipticity implies that $ L $ has closed range, and  Schauder Theory together with the method of continuity (see \cite[Theorem 5.2]{Gilbarg-Trudinger (1983)}) ensures $L$ is surjective. Hence by the Implicit Function Theorem there exists $ \epsilon >0 $ such that for every fixed $ t\in (\tau-\epsilon,\tau+\epsilon) $, equation
\[
T(\phi,t)=0
\]
has a solution $\phi$, which is additionally smooth by elliptic regularity. Therefore $ S_t\neq \emptyset $ for every $ t\in (\tau,\tau+\epsilon) $,  which contradicts the maximality of $\tau$.   \qedhere
\end{enumerate}
\end{proof}

\section{Further Developments} \label{Developments}
As a follow-up to the present work we plan to study the quaternionic Monge-Amp\`ere equation on other homogeneous spaces. 

The manifold $M_2$, for instance,  can be regarded as a $T^2$-bundle over $T^6$, so  it is quite natural to wonder whether Theorem \ref{main} might extend to $T^2$-invariant functions (instead of $T^3$-invariant). We shall next describe this setup for $M_2$ and point out the differences from the $T^3$-invariant setting considered in Theorem \ref{main}.

From \eqref{general} the quaternionic Monge-Amp\`ere equation \eqref{eq_quaternionic_Calabi_conjecture} on 
$(M_2,J_1,J_2,J_3,g)$ reduces to the following PDE on the $6$-dimensional 
base $T^6$ when the map $F$ is $T^2$-invariant
\begin{multline}\label{equationM2}
(\phi_{11}+\phi_{22}+\phi_{33}+\phi_{44}+1)(\phi_{55}+\phi_{66}+1)\\-(\phi_{35}-\phi_{26})^2-(\phi_{45}-\phi_{16})^2-
(\phi_{46}+\phi_{15})^2-(\phi_{36}+\phi_{25})^2={\rm e}^F\,,
\end{multline}
where $\varphi$ is an unknown function in $C^{\infty}(T^6)$. 
By calling 
\begin{align*}
A=\phi_{11}+\phi_{22}+\phi_{33}+\phi_{44}+1\,, && B=\phi_{55}+\phi_{66}+1
\end{align*}
and
\begin{align*}
a_1=\phi_{35}-\phi_{26}\,, && a_2=\phi_{45}-\phi_{16}\,, && a_3=\phi_{46}+\phi_{15}\,, && a_4=\phi_{36}+\phi_{25}\,,
\end{align*}
we may rewrite \eqref{equationM2} as 
\begin{equation}\label{eqn:107}
AB-\sum_{i=1}^4 a_i^2=\mathrm e^F\,.
\end{equation}
The above is elliptic and 
\begin{multline}\label{eqn:109}
B(\xi_1^2+\xi_2^2+\xi_3^2+\xi_4^2)+A(\xi_5^2+\xi_6^2)
-2 a_1(\xi_3\xi_5-\xi_2\xi_6)-
2 a_2(\xi_4\xi_5-\xi_1\xi_6)
\\-2 a_3(\xi_4\xi_6+\xi_1\xi_5)-2 a_4 (\xi_3\xi_6+\xi_2\xi_5)> 0,
\end{multline}
for every $\xi\in \R^6$, $\xi \neq 0$. 

In order to show that \eqref{equationM2} can be solved, we need only prove an a priori $C^0$-estimate for the Laplacian of the solutions to \eqref{equationM2}. The natural approach consists in adapting the proof of Proposition \ref{prop_7} by mixing 
Lemma \ref{new_Luigi} with the ellipticity of the equation. 
In this case, however, it seems that condition \eqref{eqn:109} should be replaced with a stronger assumption, one implied by the estimate  
\begin{equation}\label{nuovanecessaria}
2(\abs{a_2a_3}+\abs{a_1a_4})<{\rm e}^F\,.
\end{equation}

Applying the Laplacian operator to both sides of \eqref{eqn:107} we get 
\begin{equation*}
B\Delta A+A\Delta B+2\nabla A{\cdot}\nabla B-2\sum_{k=1}^{4} \left(\abs{\nabla a_k}^2+a_k \Delta a_k\right)
=\Delta \mathrm{e}^F,
\end{equation*}
which readily implies 
\begin{equation}
\label{eq_08}
\Delta \mathrm{e}^F\leq B\Delta A+A\Delta B+2\nabla A{\cdot}\nabla B-2\sum_{k=1}^{4} a_k \Delta a_k\,.
\end{equation}
Let $p_0$ be a maximum point for $ (\Delta \phi +2)\mathrm{e}^{-\mu \phi}$, as in Lemma \ref{new_Luigi}, and 
$$
\mu=\frac{1}{\max(\Delta \phi+2)}\,\frac{1}{1+\|\nabla \phi\|_{C^0}}\,.
$$
Using \eqref{eq_nablaPhi=0}, we see that the following relation holds at $p_0$
\[
\mu^2 (\Delta \phi+2)^2\abs{\nabla \phi}^2=\abs{\nabla \Delta \phi}^2=\abs{\nabla(A+B)}^2=\abs{\nabla A}^2 +\abs{\nabla B}^2+ 2\nabla A\cdot \nabla B\geq 2\nabla A\cdot \nabla B\,,
\]
i.e., 
\begin{equation}
\label{eq_09}
2\nabla A\cdot \nabla B\leq \mu^2 (\Delta \phi+2)^2\abs{\nabla \phi}^2\,. 
\end{equation}
To produce an upper bound for $B\Delta A+A\Delta B-2\sum_{k=1}^{4} a_k \Delta a_k$ we consider $\eta_{ij}$ as in Lemma \ref{new_Luigi} and 
$$
\xi_{i}=\sqrt{\eta_{ii}}\,.
$$
Then at $p_0$ we have  
$$
\xi_i\xi_j\geq |\eta_{ij}|\,.
$$
Moreover, 
\begin{equation*}
\begin{split}
|a_1|(\xi_3\xi_5+\xi_2\xi_6)&\geq |a_1|
\Bigl\{\abs{\mu(\Delta \phi+2)(\phi_{35}+\mu \phi_3\phi_5)- \Delta \phi_{35}}+
\abs{\mu(\Delta \phi+2)(\phi_{26}+\mu \phi_2\phi_6)- \Delta \phi_{26}}\Bigr\}\\
&\geq a_1
\Bigl\{\mu(\Delta \phi+2)(\phi_{35}+\mu \phi_3\phi_5)- \Delta \phi_{35}-
\mu(\Delta \phi+2)(\phi_{26}+\mu \phi_2\phi_6)+ \Delta \phi_{26}\Bigr\}
\\ &=\mu(\Delta \phi+2)(a_1^2+a_1\mu (\phi_3\phi_5-\phi_2\phi_6) )-a_1\Delta a_1
\end{split}
\end{equation*}
at $p_0$, i.e.,
$$
|a_1|(\xi_3\xi_5+\xi_2\xi_6)\geq \mu(\Delta \phi+2)(a_1^2+\mu a_1 (\phi_3\phi_5-\phi_2\phi_6) )-a_1\Delta a_1
$$
at $p_0$. Similarly,
\begin{align*}
\abs{a_2}(\xi_4\xi_5+\xi_1\xi_6)&\geq \mu(\Delta \phi+2)(a_2^2+\mu a_2(\phi_4\phi_5-\phi_1\phi_6) )-a_2\Delta a_2\,,\\
\abs{a_3}(\xi_4\xi_6+\xi_1\xi_5)&\geq \mu(\Delta \phi+2)(a_3^2+\mu a_3(\phi_4\phi_6+\phi_1\phi_5) )-a_3\Delta a_3\,,\\
\abs{a_4}(\xi_3\xi_6+\xi_2\xi_5)&\geq \mu(\Delta \phi+2)(a_4^2+\mu a_4(\phi_3\phi_6+\phi_2\phi_5) )-a_4\Delta a_4\,,
\end{align*}
at $p_0$. If we add up the last four inequalities and use \eqref{eqn:109} with $\xi_k=\phi_k$ for $ k=1,\dots,4 $ and $ \xi_5=-\phi_5 $, $ \xi_6=-\phi_6 $, we end up with 
$$
\begin{aligned}
& 2|a_1|(\xi_3\xi_5+\xi_2\xi_6)+2\abs{a_2}(\xi_4\xi_5+\xi_1\xi_6)+2\abs{a_3}(\xi_4\xi_6+\xi_1\xi_5)+2\abs{a_4}(\xi_3\xi_6+\xi_2\xi_5)\geq\\
& \mu(\Delta \phi+2)\left(\sum_{k=1}^4(2a_k^2-\mu B\phi_{k}^2)-\mu A(\phi_5^2+\phi_6^2)\right)-2\sum_{k=1}^4a_k\Delta a_k
\end{aligned}
$$  
at $p_0$. 

To handle the last inequality we need the following estimate 
\begin{multline}\label{quella che manca}
B(\xi_1^2+\xi_2^2+\xi_3^2+\xi_4^2)+A(\xi_5^2+\xi_6^2)\geq \\
2|a_1|(\xi_3\xi_5+\xi_2\xi_6)+2\abs{a_2}(\xi_4\xi_5+\xi_1\xi_6)+2\abs{a_3}(\xi_4\xi_6+\xi_1\xi_5)+2\abs{a_4}(\xi_3\xi_6+\xi_2\xi_5).
\end{multline}
Notice this is stronger than \eqref{eqn:109}. 

In fact, if we assume  \eqref{quella che manca}, then 
$$
\begin{aligned}
& B\sum_{k=1}^4\xi_k^2+A(\xi_5^2+\xi_6^2)\geq \mu(\Delta \phi+2)\left(\sum_{k=1}^4(2a_k^2-\mu B\phi_{k}^2)-\mu A(\phi_5^2+\phi_6^2)\right)-2\sum_{k=1}^4a_k\Delta a_k
\end{aligned}
$$  
at $p_0$ and, keeping in mind the definition of $\xi_k$, 
\begin{multline*}
B\sum_{k=1}^4\xi_k^2+A(\xi_5^2+\xi_6^2)=
\mu(\Delta \phi +2)\left( A\sum_{k=5}^{6} (\phi_{kk}+\mu \phi_k^2)+B\sum_{k=1}^{4} (\phi_{kk}+\mu \phi_k^2) \right)
-A\Delta B-B\Delta A\,,
\end{multline*}
at $p_0$. 

Therefore
$$
\begin{aligned}
& \mu(\Delta \phi +2)\left( A\sum_{k=5}^{6} (\phi_{kk}+\mu \phi_k^2)+B\sum_{k=1}^{4} (\phi_{kk}+\mu \phi_k^2) \right)
-A\Delta B-B\Delta A \geq\\ 
& \mu(\Delta \phi+2)\left(\sum_{k=1}^4(2a_k^2-\mu B\phi_{k}^2)-\mu A(\phi_5^2+\phi_6^2)\right)-2\sum_{k=1}^4a_k\Delta a_k\,, 
\end{aligned}
$$
at $p_0$, which implies 
$$
\begin{aligned}
& A\Delta B+B\Delta A
-2\sum_{k=1}^4a_k\Delta a_k\le
\\&\quad
\le\mu(\Delta \phi +2)\left( A\sum_{k=5}^{6} (\phi_{kk}+2\mu \phi_k^2)+B\sum_{k=1}^{4} (\phi_{kk}+2\mu \phi_k^2) 
-2\sum_{k=1}^4 a_k^2\right)
\\&\quad
\le \mu(\Delta \phi +2)\left(2AB-(A+B)+2\mu (A+B)\abs{\nabla \phi}^2 
-2\sum_{k=1}^4 a_k^2\right)
\\&\quad
= \mu(\Delta \phi +2)\left( 2\mathrm e^F-(\Delta \phi+2)+2\mu (\Delta\phi+2)\abs{\nabla \phi}^2\right)\,,
\end{aligned}
$$
at $p_0$. In other terms, 
\begin{equation}
\label{eq_010}
A\Delta B+B\Delta A
-2\sum_{k=1}^4a_k\Delta a_k\le \mu(\Delta \phi +2)\left( 2\mathrm e^F-(\Delta \phi+2)+2\mu (\Delta\phi+2)\abs{\nabla \phi}^2\right)
\end{equation}
at $p_0$. 
From \eqref{eq_08}, \eqref{eq_09} and \eqref{eq_010} we finally deduce
\begin{equation*}
\mu(\Delta\phi+2)^2\le -\Delta \mathrm e^F+  2\mu(\Delta \phi +2)\mathrm e^F+3\mu^2(\Delta \phi+2)^2 \abs{\nabla \phi}^2,\,
\end{equation*}
at $p_0$. At this juncture the a priori $C^0$-estimate for $\Delta \phi$ can be obtained as we did in the second part of Section \ref{Laplacian}. 

Let us point out that requiring \eqref{quella che manca} for every $\xi\in \R^6$ is equivalent to \eqref{nuovanecessaria}. Indeed the quadratic form 
\begin{multline*}
Q(\xi)=B(\xi_1^2+\xi_2^2+\xi_3^2+\xi_4^2)+A(\xi_5^2+\xi_6^2)\\
-2|a_1|(\xi_3\xi_5+\xi_2\xi_6)-2\abs{a_2}(\xi_4\xi_5+\xi_1\xi_6)-2\abs{a_3}(\xi_4\xi_6+\xi_1\xi_5)-2\abs{a_4}(\xi_3\xi_6+\xi_2\xi_5)
\end{multline*}
has matrix
\begin{equation*}
\begin{pmatrix}
\phantom{-}B & \phantom{-}0 & \phantom{-}0 & \phantom{-}0 & -\abs{a_3} & -\abs{a_2} \\
\phantom{-}0 & \phantom{-}B & \phantom{-}0 & \phantom{-}0 & -\abs{a_4} & -\abs{a_1} \\
\phantom{-}0 & \phantom{-}0 & \phantom{-}B & \phantom{-}0 & -\abs{a_1} & -\abs{a_4} \\
\phantom{-}0 & \phantom{-}0 & \phantom{-}0 & \phantom{-}B & -\abs{a_2} & -\abs{a_3} \\
-\abs{a_3} & -\abs{a_4} & -\abs{a_1} & -\abs{a_2} & \phantom{-}A & \phantom{-}0 \\
-\abs{a_2} & -\abs{a_1} & -\abs{a_4} & -\abs{a_3} & \phantom{-}0 & \phantom{-}A
\end{pmatrix}\,,
\end{equation*}
which is positive definite if and only if  
$$
B^4\left( \left(A-B^{-1}\sum_{k=1}^{4}a_k^2\right)^2- 4B^{-2}\left(\abs{a_2a_3}+\abs{a_1a_4}\right)^2 \right)>0
$$
since $B>0$. A direct computation tells that the last condition is equivalent to  
$$
2(\abs{a_2a_3}+\abs{a_1a_4})<{\rm e}^F\,.
$$

\bigskip
In analogy to the above discussion, the manifold $M_1$ arises as an $S^1$-bundle over a $T^7$-torus, and the function $F$ may be chosen to be $S^1$-invariant. If so, the quaternionic Monge-Amp\`ere equation \eqref{eq_quaternionic_Calabi_conjecture} reads 
\begin{multline*}
(\phi_{11}+\phi_{22}+\phi_{33}+\phi_{44}+1)(\phi_{55}+\phi_{66}+\phi_{77}+1)\\
-(\phi_{45}-\phi_{16}-\phi_{27})^2-(\phi_{35}+\phi_{17}-\phi_{26})^2\\
-(\phi_{36}+\phi_{47}+\phi_{25})^2-(\phi_{46}-\phi_{37}+\phi_{15})^2=\mathrm{e}^F\,,
\end{multline*}
where $\varphi$ is an unknown function in $C^{\infty}(T^7)$. 

Setting 
\begin{align*}
A=\phi_{11}+\phi_{22}+\phi_{33}+\phi_{44}+1\,, && B=\phi_{55}+\phi_{66}+\phi_{77}+1
\end{align*}
and
\begin{align*}
a_1=\phi_{45}-\phi_{16}-\phi_{27}\,, && a_2=\phi_{35}+\phi_{17}-\phi_{26}\,,\\
a_3=\phi_{36}+\phi_{47}+\phi_{25}\,, && a_4=\phi_{46}-\phi_{37}+\phi_{15}\,,
\end{align*}
the equation turns into  
\begin{equation}\label{aeqn:207}
AB-\sum_{i=1}^4 a_i^2=\mathrm e^F\,.
\end{equation}
The above is elliptic and 
\begin{multline}\label{aeqn:209}
B(\xi_1^2+\xi_2^2+\xi_3^2+\xi_4^2)+A(\xi_5^2+\xi_6^2+\xi^2_7)
-2 a_1(\xi_4\xi_5-\xi_1\xi_6-\xi_2\xi_7)\\
-
2 a_2(\xi_3\xi_5+\xi_1\xi_7-\xi_2\xi_6)
-2 a_3(\xi_3\xi_6+\xi_4\xi_7+\xi_2\xi_5)-2 a_4 (\xi_4\xi_6-\xi_3\xi_7+\xi_1\xi_5)> 0\,,
\end{multline}
for every $\xi\in \R^7$, $\xi \neq 0$. 

We proceed as in the previous case, and  choose $p_0$ and  $\eta_{ij}$ as in Lemma \ref{new_Luigi} and 
$$
\mu=\frac{1}{\max(\Delta \phi+2)}\,\frac{1}{1+\|\nabla \phi\|_{C^0}}\,,
$$
resulting in 
\begin{equation*}
\Delta \mathrm{e}^F\leq B\Delta A+A\Delta B+ \mu^2 (\Delta \phi+2)^2\abs{\nabla \phi}^2-2\sum_{k=1}^{4} a_k \Delta a_k\,, \quad \mbox{ at }p_0\,.
\end{equation*}

Set $\xi_i=\sqrt{\eta_{ii}}$  and apply Lemma \ref{new_Luigi} to obtain 
\begin{equation*}
\begin{split}
|a_1|&(\xi_4\xi_5+\xi_1\xi_6+\xi_2\xi_7)\\
\geq& |a_1|
\Bigl\{\abs{\mu(\Delta \phi+2)(\phi_{45}+\mu \phi_4\phi_5)- \Delta \phi_{45}}+
\abs{\mu(\Delta \phi+2)(\phi_{16}+\mu \phi_1\phi_6)- \Delta \phi_{16}}\\
&+\abs{\mu(\Delta \phi+2)(\phi_{27}+\mu \phi_2\phi_7)- \Delta \phi_{27}}\Bigr\}\\
\geq& a_1
\Bigl\{\mu(\Delta \phi+2)(\phi_{45}+\mu \phi_4\phi_5)- \Delta \phi_{45}-
\mu(\Delta \phi+2)(\phi_{16}+\mu \phi_1\phi_6)+ \Delta \phi_{16}\\
&-\mu(\Delta \phi+2)(\phi_{27}+\mu \phi_2\phi_7)+ \Delta \phi_{27}\Bigr\}
\\ =&\mu(\Delta \phi+2)(a_1^2+a_1\mu (\phi_4\phi_5-\phi_1\phi_6-\phi_2\phi_7) )-a_1\Delta a_1
\end{split}
\end{equation*}
at $p_0$, i.e.
$$
|a_1|(\xi_4\xi_5+\xi_1\xi_6+\xi_2\xi_7)\geq \mu(\Delta \phi+2)(a_1^2+\mu a_1 (\phi_4\phi_5-\phi_1\phi_6-\phi_2\phi_7) )-a_1\Delta a_1
$$
at $p_0$. From that we deduce
\begin{align*}
\abs{a_2}(\xi_3\xi_5+\xi_1\xi_7+\xi_2\xi_6)&\geq \mu(\Delta \phi+2)(a_2^2+\mu a_2(\phi_3\phi_5+\phi_1\phi_7-\phi_2\phi_6) )-a_2\Delta a_2\,,\\
\abs{a_3}(\xi_3\xi_6+\xi_4\xi_7+\xi_2\xi_5)&\geq \mu(\Delta \phi+2)(a_3^2+\mu a_3(\phi_3\phi_6+\phi_4\phi_7+\phi_2\phi_5) )-a_3\Delta a_3\,,\\
\abs{a_4}(\xi_4\xi_6+\xi_3\xi_7+\xi_1\xi_5)&\geq \mu(\Delta \phi+2)(a_4^2+\mu a_4(\phi_4\phi_6-\phi_3\phi_7+\phi_1\phi_5) )-a_4\Delta a_4\,,
\end{align*}
at $p_0$. The sum of the previous four inequalities, together with \eqref{aeqn:209},  yields
$$
\begin{aligned}
& 2|a_1|(\xi_4\xi_5+\xi_1\xi_6+\xi_2\xi_7)+
2\abs{a_2}(\xi_3\xi_5+\xi_1\xi_7+\xi_2\xi_6)\\
&+
2\abs{a_3}(\xi_3\xi_6+\xi_4\xi_7+\xi_2\xi_5)+
2\abs{a_4}(\xi_4\xi_6+\xi_3\xi_7+\xi_1\xi_5)\geq\\
& \mu(\Delta \phi+2)\left(\sum_{k=1}^4(2a_k^2-\mu B\phi_{k}^2)-\mu A\sum_{k=5}^7\phi_k^2\right)-2\sum_{k=1}^4a_k\Delta a_k
\end{aligned}
$$  
at $p_0$. 

We need the following estimate 
\begin{multline}\label{quella che manca2}
B(\xi_1^2+\xi_2^2+\xi_3^2+\xi_4^2)+A(\xi_5^2+\xi_6^2+\xi^2_7)
-2 |a_1|(\xi_4\xi_5+\xi_1\xi_6+\xi_2\xi_7)\\
-2 |a_2|(\xi_3\xi_5+\xi_1\xi_7+\xi_2\xi_6)
-2 |a_3|(\xi_3\xi_6+\xi_4\xi_7+\xi_2\xi_5)-2 |a_4| (\xi_4\xi_6+\xi_3\xi_7+\xi_1\xi_5)> 0\,,
\end{multline}
at $p_0$, which is stronger than \eqref{aeqn:209}. Once this has been established, the result follows. 

To prove \eqref{quella che manca2} one has to show that the quadratic form 
\begin{multline*}
Q(\xi)=B(\xi_1^2+\xi_2^2+\xi_3^2+\xi_4^2)+A(\xi_5^2+\xi_6^2+\xi^2_7)
-2 |a_1|(\xi_4\xi_5+\xi_1\xi_6+\xi_2\xi_7)\\-
2 |a_2|(\xi_3\xi_5+\xi_1\xi_7+\xi_2\xi_6)-2 |a_3|(\xi_3\xi_6+\xi_4\xi_7+\xi_2\xi_5)-2 |a_4| (\xi_4\xi_6+\xi_3\xi_7+\xi_1\xi_5)
\end{multline*}
on $\R^7$ is positive-definite. This is equivalent to demanding two things:
$$
\mathrm{e}^{2F}-4(\abs{a_2a_3}+\abs{a_1a_4})^2>0\,,
$$

\[
\begin{split}
&\mathrm{e}^{3F}-4\mathrm{e}^F\left( \left(\abs{a_2a_3}+\abs{a_1a_4}\right)^2+\left(\abs{a_1a_3}+\abs{a_2a_4}\right)^2 +\left(\abs{a_1a_2}+\abs{a_3a_4}\right)^2 \right) \\
&-16\left(\abs{a_2a_3}+\abs{a_1a_4}\right)\left(\abs{a_1a_3}+\abs{a_2a_4}\right)\left(\abs{a_1a_2}+\abs{a_3a_4}\right)>0\,.
\end{split}
\]

\bigskip 
We wrap up this overview of our future plans by observing that there exist torus fibrations whose hypercomplex structure is not locally trivial. On these spaces Alesker's Theorem cannot be applied, so once the $C^{0}$-estimate of the Laplacian is at hand one needs to prove the $C^{2,\alpha}$-estimate by alternative arguments.     

We expect that the study of the equation on these explicit examples will give new insight for the handling of the general case.

\end{document}